\documentclass[11pt]{article}
\usepackage[numbers,sort&compress]{natbib}
\usepackage{enumerate}
\usepackage{amscd}
\usepackage{amsmath}
\usepackage{latexsym}
\usepackage{amsfonts}
\usepackage{amssymb}
\usepackage{amsthm}
\usepackage{verbatim}
\usepackage{mathrsfs}
\usepackage{enumerate}
\usepackage{hyperref}

\theoremstyle{plain}\newtheorem{definition}{Definition}[section]
\theoremstyle{definition}\newtheorem{theorem}{Theorem}[section]
\theoremstyle{plain}\newtheorem{lemma}[theorem]{Lemma}
\theoremstyle{plain}\newtheorem{coro}[theorem]{Corollary}
\theoremstyle{plain}
\theoremstyle{remark}\newtheorem{remark}{Remark}[section]
\usepackage{xcolor}

\newcommand{\Div}{\mathrm{div}\,}
\newcommand{\B}{\Big}

\newcommand{\be}{\begin{equation}}
\newcommand{\ee}{\end{equation}}
 \newcommand{\ba}{\begin{aligned}}
 \newcommand{\ea}{\end{aligned}}

  \newcommand{\f}{\frac}
    
  \newcommand{\ben}{\begin{enumerate}}
   \newcommand{\een}{\end{enumerate}}

\newcommand{\Rmnum}[1]{\expandafter\@slowromancap\romannumeral #1@}

\allowdisplaybreaks

\numberwithin{equation}{section}
\begin{document}
%%%%%%%%%%%%%%%%%%%%%%%%%%%%%%%%%%%%%%%%%%%%%%%%%%%%%%%%%%%%%%%%%%%%%%%%%%%%%%%%%%%%%%%%%%%%%%%%%%%%
\title{Energy equality   in     the isentropic compressible
Navier-Stokes equations allowing vacuum }
\author{Yulin Ye\footnote{School of Mathematics and Statistics,
Henan University,
Kaifeng, 475004,
P. R. China. Email: ylye@vip.henu.edu.cn},
\;~   Yanqing Wang\footnote{Corresponding author.  College of Mathematics and   Information Science, Zhengzhou University of Light Industry, Zhengzhou, Henan  450002,  P. R. China Email: wangyanqing20056@gmail.com}~  and\, Wei Wei\footnote{School of Mathematics and Center for Nonlinear Studies, Northwest University, Xi'an, Shaanxi 710127,  P. R. China  Email: ww5998198@126.com }
}
\date{}
\maketitle
\begin{abstract}

It is well-known that  a Leray-Hopf weak solution in   $L^4 (0,T; L^4(\Omega))$  for the incompressible Navier-Stokes system is persistence of energy due to Lions \cite{[Lions]}. In this paper, it is  shown that
  Lions's   condition  for energy balance is also valid for the   weak solutions of  the  isentropic compressible Navier-Stokes equations  allowing vacuum
under suitable integrability conditions on the density and its derivative.
This allows us to establish
  various
sufficient conditions implying energy equality for the compressible flow as well as
  the non-homogenous incompressible  Navier-Stokes equations. This is an improvement     of corresponding results obtained by
Yu in \cite[ Arch. Ration. Mech. Anal., 225 (2017)]{[Yu2]}, and our criterion via the gradient of the velocity partially
 answers a question posed by Liang in
 \cite[Proc. Roy. Soc.
Edinburgh Sect. A (2020)]{[Liang]}.
  \end{abstract}
\noindent {\bf MSC(2000):}\quad 35Q30, 35Q35, 76D03, 76D05\\\noindent
{\bf Keywords:} compressible Navier-Stokes equations;  energy equality;  vacuum %%%%%%%%%%
\section{Introduction}
\label{intro}
\setcounter{section}{1}\setcounter{equation}{0}
 In this paper, we are concerned with the classical isentropic compressible Navier-Stokes equations in a periodic domain $\Omega=\mathbb{T}^{d}$ with $d=2,3$
\be\left\{\ba\label{INS}
&\rho_t+\nabla \cdot (\rho v)=0, \\
&(\rho v)_{t} +\Div(\rho v\otimes v)+\nabla
P(\rho )- \mu\Delta v- (\mu+\lambda) \nabla\text{div\,}v=0,\\
\ea\right.\ee
 where $\rho\geq 0$ stands for the density of the flow, $v$
represent the fluid  velocity field and $P(\rho)=\rho^{\gamma}$ with $\gamma >1$ is the scalar pressure.
The viscosity coefficients $\mu$ and $\lambda$ satisfy $\mu\geq 0$ and $2\mu+d\lambda>0$. This model usually describes the motion of a compressible viscous barotropic fluid.
We complement equations \eqref{INS} with initial data
\begin{equation}\label{INS1}
	\rho(0,x)=\rho_0(x),\ (\rho v)(0,x)=(\rho_0 v_0)(x),\ x\in \Omega,
\end{equation}
where we define $v_0 =0$ on the sets $\{x\in \Omega:\ \rho_0=0\}.$

Recently,
 originated from
Yu's work  \cite{[Yu2]},
 the   energy conservation
of weak solutions in the  compressible fluid equations
 attracts a lot of attention (see e.g.\cite{[Yu2],[Liang],[NNT],[NNT1],[NNT2],[CLWX],[WY],[CY],[EGSW], [ADSW]}).
Towards critical regularity for a weak solution keeping energy  equality, the main result in \cite{[Yu2]} is that if a weak solution $(\rho, v)$  in the sense of Definition  \ref{wsdefi} satisfies
\be\ba\label{yu}
0\leq\rho\leq c<\infty, \ \ \nabla\sqrt{\rho} \in L^{\infty}(0,T;L^{2}(\mathbb{T}^{d})),\\
v\in L^{p}(0,T;L^{q}(\mathbb{T}^{d})) ~~p\geq4   \ \text{ and}\   q\geq6,\ and\ v_0\in L^{q_0},\ q_0\geq 3,
\ea
\ee
  the following energy equality is valid
     \begin{equation}\label{EI}\ba
\int_{\mathbb{T}^{d}}\left( \frac{1}{2}\rho |v|^2+\f{\rho^{\gamma}}{\gamma-1} \right) dx +\int_0^T\int_{\mathbb{T}^{d}}&\left[\mu|\nabla v|^2+(\mu+\lambda)|\Div v|^2 \right] dxdt\\=&\int_{\mathbb{T}^{d}}\left( \frac{1}{2} \rho_0 |v_0|^2+\f{\rho_{0}^{\gamma}}{\gamma-1}\right)dx.
\ea
\end{equation}
The famous Onsager's conjecture \cite{[Onsager]} is what is the critical regularity for a weak solution conserving energy in the
incompressible  Euler equations. In this direction, a celebrated result due to Constantin-E-Titi \cite{[CET]} states that a weak solution of the Euler system conserves the energy if $v$ is in $L^{3}(0,T; B^{\alpha}_{3,\infty})$ with $\alpha>1/3.$
The  negative results of Onsager's conjecture can be found in \cite{[BDSV],[Isett]}.
For the 3D incompressible Navier-Stokes,   a kind of well-known positive
energy conservation class is
Lions-Shinbrot type criterion in terms of integrability of the velocity $v$. Particularly, in \cite{[Lions]}, Lions showed that a weak solution
$v$   satisfies the energy equality in the  incompressible Navier-Stokes in any spatial dimension if the velocity $v$ meets
\be\label{lions}
v\in L^{4}(0,T;L^{4}(\mathbb{T}^d)).\ee
Lions's energy conservation criterion  was  generalized by Shinbrot  in \cite{[Shinbrot]} to
\be\label{Shinb}
v\in L^{p}(0,T;L^{q}(\mathbb{T}^d)),~\text{with}~\f{2}{p}+
 \f{2}{q}=1~\text{and}~q\geq 4.\ee
For $q<4$, Lions-Shinbrot type criterion in tri-dimensional space states
\be\label{TBY}
v\in L^{p}(0,T;L^{q}(\mathbb{T}^3)), ~\text{with}~ \f{1}{p}+
 \f{3}{q}=1 ~\text{and}~3<q< 4,
 \ee
which was studied by Taniuchi \cite{[Taniuchi]} and Beirao da Veiga-Yang \cite{[BY]}. Recent progress about Lions-Shinbrot type criterion  can be found in \cite{[CCFS],[CL],[BC],[Zhang],[WY],[Yu1],[Yu3],[WWY]}.

   Now, we turn attention back to compressible fluid.
 For equations \eqref{INS} posed an open  bounded domain $\Omega$ with  Dirichlet boundary condition, Chen-Liang-Wang-Xu \cite{[CLWX]} showed that
 \eqref{yu} yields  the energy balance law \eqref{EI}.
 For the general  compressible models including \eqref{INS}  in the absence of vacuum,  it is shown that
the following  conditions
 \be\ba\label{NNT}
 &0<  c_{1}\leq\rho\leq c_{2}<\infty,  v\in L^{\infty}(0,T; L^{2}
 (\mathbb{T}^{d})), \nabla v\in L^{2}(0,T; L^{2}(\mathbb{T}^{d})),   \\&\sup_{t\in(0,T)}\sup_{|h|<\varepsilon}|h|^{-\f12}
 \|\rho(\cdot+h,t)-\rho(\cdot,t)\|_{L^{2}(\mathbb{T}^{d})}<\infty,\\
 &v\in L^{4}(0,T;L^{4}(\mathbb{T}^{d})).
 \ea\ee
 ensure  that a  weak solution meets the energy  equality by
 Nguye-Nguye-Tang in \cite{[NNT]}.
The energy conservation criterion $\eqref{NNT}$ is still valid if both   $\eqref{NNT}_{2}$ and $\eqref{NNT}_{3}$ are replaced by
      \eqref{Shinb} or \eqref{TBY}, which was recently proved in \cite{[WY]}. We also refer the readers to \cite{[Liang],[WY]} for energy  equality class via the gradient of the velocity.
      Roughly speaking, from  the aforementioned works, the  Lions-Shinbrot type  energy
  conservation class    for the isentropic compressible Navier-Stokes equations  in the case away
from vacuum is the same as the incompressible Naiver-Stokes equations.  A natural question is weather this type  energy
conservation class is valid when
the
initial data of system \eqref{INS} contain vacuums.
The objective of this paper   will be devoted to
this issue.

We state our first result   as follows.
\begin{theorem}\label{the1.1}
The energy equality \eqref{EI} of   weak solutions  $(\rho, v)$ to the isentropic compressible Navier-Stokes equations \eqref{INS}-\eqref{INS1} holds if one of the following four conditions is satisfied
 \begin{enumerate}[(1)]
 \item $0 \leq\rho\leq c_{2}<\infty$, $\nabla\sqrt{\rho}\in L^{4}(0,T;L^{4}(\mathbb{T}^{d})),$  $ \nabla v\in L^{2}(0,T; L^{2}(\mathbb{T}^{d}))$,  $v_0\in L^2(\mathbb{T}^{d})$,
 \be\label{vei}
v\in L^{4}(0,T;L^{4}(\mathbb{T}^{d}));\ee
 \item $0 \leq\rho\leq c_{2}<\infty$, $\nabla\sqrt{\rho}\in L^{4}(0,T;L^{4}(\mathbb{T}^{d})),$  $ \nabla v\in L^{2}(0,T; L^{2}(\mathbb{T}^{d}))$,  $v_0\in L^2(\mathbb{T}^{d})$,\be\label{vei819}
v\in L^{p}(0,T;L^{q}(\mathbb{T}^{d}))~ \text{with} ~ \f{1}{p}+
\f{3}{q}=1, 3<q\leq  4; \ee
\item \label{2vei2} $0 \leq\rho\leq c_{2}<\infty$,    $\nabla\sqrt{\rho}\in L^{4}(0,T; L^{2}(\mathbb{T}^{d}))$, $ \nabla v\in L^{2}(0,T; L^{2}(\mathbb{T}^{d}))$,  $v_0\in L^3(\mathbb{T}^{d})$
 \be\label{vei2}
v\in L^{4}(0,T;L^{6}(\mathbb{T}^{d}));   \ee
 \item $0 \leq\rho\leq c_{2}<\infty$, $\nabla\sqrt{\rho}\in L^{4}(0,T;L^{4}(\mathbb{T}^{d})),$       and $ \nabla v\in L^{2}(0,T; L^{2}(\mathbb{T}^{d}))$,  $v_0\in L^2(\mathbb{T}^{d})$
 \be\label{tivei}
  \nabla v\in L^{p}(0,T;L^{r}(\mathbb{T}^{d}))  ~\text{with}~ \f{1}{p}+
  \f{3}{r}=1+\f{3}{d}, \f{3d}{d+3}<r\leq \f{4d}{d+4}.
 \ee
  \end{enumerate}
\end{theorem}
\begin{remark}
This theorem extends the well-known Lions's energy equality criterion to the isentropic compressible equations \eqref{INS} with vacuum.
Sufficient conditions \eqref{TBY} implying energy equality  are also   generalized from the incompressible fluid to the  compressible
flow allowing vacuum.
\end{remark}
 \begin{remark}
 Compared with Yu' class \eqref{yu}, the condition \eqref{vei2}
   relaxes the  integrability of the density.
\end{remark}
\begin{remark}
 In \cite[Remark 1.7, Proc. Roy. Soc. Sect. A (2020)]{[Liang]}, Liang mentioned a question that one considers the energy equality sufficient condition via the gradient of the velocity when the vacuum state is allowed. From \eqref{tivei}, we see that  this result gives an affirmative answer to this question.
\end{remark}
 \begin{remark}
By means of a global mollification combined with an independent boundary cut-off used in  \cite{[CLWX]}, one may generalize Theorem \ref{the1.1}   from periodic boundary to Dirichlet boundary.
\end{remark}
We give some comments on the proof of this theorem.
In contrast  with the works of \cite{[WY],[NNT]}  without vacuum, it seems that
the test function and the Constantin-E-Titi type commutators on mollifying kernel   break  down in the presence of vacuum. The test function as \cite{[Yu2]}, a Lions's type commutator involving space-time mollifier and the
  strategy for energy equality condition in
terms of both the velocity and its gradient in \cite{[WY]} help us prove Theorem \ref{the1.1} and deal with the vacuum, (see also the Theorem \ref{the1.2}). It is remarking that a new ingredient
  is that we present  a generalized  Lions's type commutators involving the  integral norms with different exponents in space and
time, and outline its proof, see Lemma  \ref{pLions} for the details.
Moreover, since the initial vacuum is allowed, one just obtains the energy  $\sqrt{\rho}v\in L^{\infty}(0,T;L^{2}(\Omega))$ rather than $v\in L^{\infty}(0,T;L^{2}(\Omega))$.
It is  the main difficulty caused by this   to derive the blow up criterion of a strong solution of system with vacuum in terms of the gradient of the velocity.   We refer the readers to \cite{[JWY]} for the  detail.
In \cite[Remark 5.1, p4]{[Lions2]},
Lions mentioned that one can get further regularity $v\in L^{2}(0,T;L^{2}(\Omega))$ in equations \eqref{INS} for the periodic case via careful Poincar\'e type inequality. However, it seems that it is difficult to use  $v\in L^{2}(0,T;L^{2}(\Omega))$ in the derivation of \eqref{tivei}.
 To prove \eqref{tivei} and deal with the vacuum, in the spirit of   \cite[Remark 5.1, p4]{[Lions2]}, we make full use of $\nabla v\in L^{p}(0,T;L^{q}(\mathbb{T}^{3}))$ and apply the generalized Lions type Poincar\'e inequality to overcome the  difficulty. Based on this, roughly speaking, we observe that  $\nabla v\in L^{p}(0,T;L^{q}(\mathbb{T}^{3}))$ means $\int_{\Omega}vdx\in L^{p}(0,T)$. This is the critical point to get \eqref{tivei}.

As a by-product of above theorem, one further  derives various energy conservation class for the three-dimensional  isentropic compressible Navier-Stokes equations \eqref{INS}.
\begin{coro}\label{coro1.2}
Let $(\rho,v)$ be a weak solution of isentropic compressible Navier-Stokes equations \eqref{INS}-\eqref{INS1} in the sense of Definition \ref{wsdefi}. Assume  the density $\rho$ satisfies one of the following condition
\begin{enumerate}[(1)]
 \item $0 \leq\rho\leq c_{2}<\infty, \nabla\sqrt{\rho}\in L^{4}(0,T;L^{4}(\mathbb{T}^{3}));$
 \item $\nabla\sqrt{\rho}\in L^{\infty}(0,T;L^{4}(\mathbb{T}^{3}))$;
\item $0\leq\rho\leq c<\infty,$ $ \nabla\sqrt{\rho}\in L^{\infty}(0,T;L^{2}(\mathbb{T}^{3}))$ and
$\nabla\sqrt{\rho}\in L^{p_{1}}(0,T;L^{q_{1}}(\mathbb{T}^{3})), with\f{1}{p_{1}}+\f{1}{q_{1}}=\f12,q_{1}\geq4$;
\end{enumerate}
 and the velocity meet one of the following
\begin{enumerate}[(i)]
\item  $ \nabla v\in L^{2}(0,T; L^{2}(\mathbb{T}^{3}))$
$
v\in L^{p}(0,T;L^{q}(\mathbb{T}^{3}))  ~\text{with}~ \f{1}{p}+
\f{3}{q}=1, 3<q\leq 4,\ v_0\in L^2(\mathbb{T}^{3});  $
\item
$ \nabla v\in L^{2}(0,T; L^{2}(\mathbb{T}^{3}))$
 $
   \nabla v\in L^{p}(0,T;L^{r}(\mathbb{T}^{3}))  ~\text{with}~ \f{1}{p}+
   \f{3}{r}=2, \ \f{3}{2}<r\leq \f{12}{7},\ v_0\in L^2(\mathbb{T}^{3}).
$
\end{enumerate}
\end{coro}

 As mentioned above, the key point for the proof of Theorem \ref{the1.1} is the following energy equality condition in terms of the velocity and its gradient. Indeed, Theorem \ref{the1.1} is an immediate consequence of the following theorem.
\begin{theorem}\label{the1.2}
 For any $p,q\geq 4$ and $dp< 2q+3d$ with $d\geq 2$, the energy equality \eqref{EI} of   weak solutions  $(\rho, v)$ to the isentropic compressible  Navier-Stokes equation \eqref{INS}-\eqref{INS1} is valid provided
 \be\label{key0}
\ba
&0\leq \rho<c<\infty, \nabla\sqrt{\rho}\in L^{\f{p}{p-3}}(0,T;L^{\f{q}{q-3}}(\mathbb{T}^{d})),\\& v\in L^{p}(0,T;L^{q}(\mathbb{T}^{d})), \nabla v\in L^{\f{p}{p-2}}(0,T;L^{\f{q}{q-2}}(\mathbb{T}^{d})),\ v_0\in L^{\f{q}{2}}(\mathbb{T}^{d}).
\ea\ee
\end{theorem}

Finally, following the path of the above theorem, one can consider the non-homogenous incompressible Navier-Stokes equations below
\be\left\{\ba\label{nhins}
&\rho_{t} + \Div (\rho v)   =0,   \\
&(\rho v)_{t} +\Div (\rho v\otimes v)-\mu\Delta v +\nabla P = 0,\\
&\Div v=0,\\
\ea\right.\ee
and derive a similar result as
Theorem \ref{the1.2}. Since the velocity is divergence-free, the pressure part in the proof becomes trivial.
 Hence, we just formulate the corresponding result here.
 \begin{theorem}\label{the1.3}
 Let  $(\rho, v)$ be a weak solution  to the Navier-Stokes equation \eqref{nhins}. Assume that the pair $(\rho, v)$ obeys
 \be\label{vein}\ba
  &0 \leq\rho\leq c_{2}<\infty, \nabla\sqrt{\rho}\in L^{4}(0,T;L^{4}(\mathbb{T}^{3})), \nabla v\in L^{2}(0,T; L^{2}(\mathbb{T}^{3})),\\
&v\in L^{p}(0,T;L^{q}(\mathbb{T}^{3}))  ~\text{with}~ \f{1}{p}+
\f{3}{q}=1, 3<q\leq 4, ~\text{and}~ v_0\in L^2(\mathbb{T}^3); \ea\ee
 then,
the energy equality below holds
  \begin{equation}  \label{EIN}
\int_{\mathbb{T}^{3}} \frac{1}{2}\rho |v|^2dx  +\int_0^T\int_{\mathbb{T}^{3}}  \mu|\nabla v|^2  dxdt=\int_{\mathbb{T}^{3}}  \frac{1}{2} \rho_0 |v_0|^2dx.
\end{equation}
 \end{theorem}
 \begin{remark}
 When the density $\rho$ becomes a constant,   Theorem \ref{the1.3} reduces to the Lions's classical result \eqref{lions}.
 \end{remark}
 \begin{remark}
Previous results on
energy balance proved in
\cite{[LS],[NNT],[WY]} are all without vacuum. It seems that this theorem will be the first criterion on
energy equality of weak solutions to non-homogenous incompressible Navier-Stokes equations in the presence of vacuum.
 \end{remark}

The remainder of this paper is organized as follows. Section 2 is devoted to the auxiliary lemmas involving mollifier and the key inequality. In Section 3, we first  present the proof of Theorem \ref{the1.2}. Then, based on Theorem \ref{the1.2}, we complete the proof of Theorem \ref{the1.1} and Corollary \ref{coro1.2}.
\section{Notations and some auxiliary lemmas} \label{section2}

First, we introduce some notations used in this paper.
 For $p\in [1,\,\infty]$, the notation $L^{p}(0,\,T;X)$ stands for the set of measurable functions on the interval $(0,\,T)$ with values in $X$ and $\|f(t,\cdot)\|_{X}$ belonging to $L^{p}(0,\,T)$. The classical Sobolev space $W^{k,p}(\Omega)$ is equipped with the norm $\|f\|_{W^{k,p}(\Omega)}=\sum\limits_{|\alpha| =0}^{k}\|D^{\alpha}f\|_{L^{p}(\Omega)}$.  For simplicity, we denote by $$\int_0^T\int_{\Omega} f(t, x)dxdt=\int_0^T\int f\ ~~\text{and}~~ \|f\|_{L^p(0,T;X )}=\|f\|_{L^p(X)}.$$

Let $\eta$ be non-negative   smooth function supported in the space-time ball of radius 1 and its integral equals to 1. We define the rescaled space-time mollifier $\eta_{\varepsilon}(t,x)=\frac{1}{\varepsilon^{d+1}}\eta(\frac{t}{\varepsilon},\frac{x}{\varepsilon})$
$$
f^{\varepsilon}(t,x)=\int_{0}^{T}\int_{\Omega}f(s,y)\eta_{\varepsilon}(t-s,x-y)dyds.
$$

\begin{definition}\label{wsdefi}
 	A pair ($\rho,v$) is called a weak solution to \eqref{INS} with initial data \eqref{INS1} if ($\rho,v$) satisfies
 	\begin{enumerate}[(i)]
 		\item equation \eqref{INS} holds in $D'(0,T;\Omega)$ and
 		\begin{equation}P(\rho ), \rho |v|^2\in L^\infty(0,T;L^1(\Omega)),\ \ \ \nabla v\in L^2(0,T;L^2(\Omega)),
 		\end{equation}
 	\item[(ii)]
 	the density $\rho$ is a renormalized solution of \eqref{INS} in the sense of \cite{[PL]}.
 		\item[(iii)]
 		the energy inequality holds
 		\begin{equation}\label{energyineq}
 		\begin{aligned}
 	E(t) +\int_0^T\int_{\Omega}\left( \mu|\nabla v|^2+(\mu+\lambda)|\Div v|^2 \right) dxdt \leq E(0), 		\end{aligned}\end{equation}
 where $E(t)=\int_{\Omega}\left( \frac{1}{2}\rho |v|^2+\f{\rho^{\gamma}}{\gamma-1} \right) dx$.
 	\end{enumerate}
 \end{definition}
To the knowledge of the authors,  the following statement about Lions's type commutators involving space-time mollifier with $p=q, p_i=q_i$, $(i=1,2)$ can be found in \cite{[LV]}, which plays an important role in recent works \cite{[Yu2],[CY],[CLWX]} and whose complete rigorous   proof was not presented any where (see \cite{[Lions1]} for original version).   To make our work more self-contained
and more readable, we shall outline its proof.
\begin{lemma}
	\label{pLions}Let $1\leq p,q,p_1,q_1,p_2,q_2\leq \infty$,  with $\frac{1}{p}=\frac{1}{p_1}+\frac{1}{p_2}$ and $\frac{1}{q}=\frac{1}{q_1}+\frac{1}{q_2}$.
	Let $\partial$ be a partial derivative in space or time, in addition, let  $\partial_t f,\ \nabla f \in L^{p_1}(0,T;L^{q_1}(\Omega))$, $g\in L^{p_2}(0,T;L^{q_2}(\Omega))$.   Then, there holds $$\|{\partial(fg)^\varepsilon}-\partial(f\,{g}^\varepsilon)\|_{L^p(0,T;L^q(\Omega))}\leq C\left(\|\partial_{t} f\|_{L^{p_{1}}(0,T;L^{q_{1}}(\Omega))}+\|\nabla f\|_{L^{p_{1}}(0,T;L^{q_{1}}(\Omega))}\right)\|g\|_{L^{p_{2}}(0,T;L^{q_{2}}(\Omega))},
	$$
	for some constant $C>0$ independent of $\varepsilon$, $f$ and $g$. Moreover, $${\partial{(fg)^\varepsilon}}-\partial{(f\,{g^\varepsilon})}\to 0\quad\text{ in } {L^{p}(0,T;L^{q}(\Omega))},$$
	as $\varepsilon\to 0$ if $p_2,q_2<\infty.$
\end{lemma}
\begin{proof}
	First,  we let
	\begin{equation}\label{l1}
		\begin{aligned}
			\partial (fg)^\varepsilon-\partial (f g^\varepsilon)=\partial(fg)^\varepsilon-f\partial g^\varepsilon-\partial f g^\varepsilon=	G^\varepsilon-\partial f g^\varepsilon,
		\end{aligned}	
	\end{equation}	
	where $G^\varepsilon=\partial (fg)^\varepsilon-f\partial g^\varepsilon$.
	Due to the properties of the Mollifier and the integration by parts, we have
	\begin{equation}\label{l2}
		\begin{aligned}
			|G^\varepsilon (t,x)|=&\left|\int\int \left(f(s,y)-f(t,x)\right)g(s,y)\frac{1}{\varepsilon^{d+1}}\frac{1}{\varepsilon}\partial \eta (\frac{t-s}{\varepsilon},\frac{x-y}{\varepsilon})	dyds\right| \\ 			=&\left|\int\int\left[f(s,y)-f(t,y)+f(t,y)-f(t,x)\right]g(s,y)\frac{1}{\varepsilon^{d+1}}\frac{1}{\varepsilon}\partial \eta (\frac{t-s}{\varepsilon},\frac{x-y}{\varepsilon})	dyds \right| \\
			\leq & C\int_{t-\varepsilon}^{t+\varepsilon}\int_{B(x,\varepsilon)}\frac{1}{\varepsilon^{d+1}}\int_0^1|\partial_t f(t+\tau(s-t), y)|d\tau|g(s,y)| dyds\\
			&\ \ \ \ \ \ +\int_{t-\varepsilon}^{t+\varepsilon}\int_{B(x,\varepsilon)}\frac{1}{\varepsilon^{d+1}}\int_0^1|\nabla f(t,x+\tau(y-x))|d\tau|g(s,y)| dyds\\
			=&I_1+I_2,
		\end{aligned}
	\end{equation}
	here, $B(x,\varepsilon)=\{y\in \Omega; |y-x|<\varepsilon\}$ is an open ball centered at $x$. To control $I_1$ and $I_2$, by the H\"older's inequality,  we know that
	\begin{equation}\label{l3}
		\begin{aligned}
			I_1\leq& C\int_{t-\varepsilon}^{t+\varepsilon}\frac{1}{\varepsilon}\left(\int_{B(x,\varepsilon)}\frac{1}{\varepsilon^{d}}\int_0^1|\partial_t f(t+\tau(s-t),y)|^{s_1}d\tau dy\right)^{\frac{1}{s_1}}\left(\int_{B(x,\varepsilon)}\frac{1}{\varepsilon^d}|g(s,y)|^{s_2}dy\right)^{\frac{1}{s_2}} ds\\
			\leq& C\int_{t-\varepsilon}^{t+\varepsilon}\frac{1}{\varepsilon}\left(\int_{B(0,\varepsilon)}\frac{1}{\varepsilon^{d}}\int_0^1|\partial_t f(t+\tau(s-t),x-z)|^{s_1}d\tau dz\right)^{\frac{1}{s_1}}\left(\int_{B(0,\varepsilon)}\frac{1}{\varepsilon^d}|g(s,x-z)|^{s_2}dz\right)^{\frac{1}{s_2}} ds,
		\end{aligned}
	\end{equation}
	where $\frac{1}{s_1}+\frac{1}{s_2}=1, 1\leq s_1\leq q_1$ and $1\leq s_2\leq q_2$.
	Moreover,  notice that
	\begin{equation}\label{l4}
		\begin{aligned}
			&\int_{B(0,\varepsilon)}\frac{1}{\varepsilon^{d}}\int_0^1|\partial_t f(t+\tau(s-t),x-z)|^{s_1}d\tau dz\\
			=&\int_{\mathbb{R}^d}\int_0^1|\partial_t f(t+\tau(s-t),x-z)|^{s_1}d\tau \frac{{\bf 1}_{B(0,\varepsilon)}(z)}{\varepsilon^{d}} dz\\
			=&\left(\int_0^1 |\partial_t f(t+\tau(s-t))|^{s_1}d\tau* J_\varepsilon\right)(x),
		\end{aligned}
	\end{equation}
	and
	\begin{equation}\label{l5}
		\begin{aligned}
			&\int_{B(0,\varepsilon)}\frac{1}{\varepsilon^d}|g(s,x-z)|^{s_2}dz\\
			=&\int_{\mathbb{R}^d}|g(s,x-z)|^{s_2}\frac{{\bf 1}_{B(0,\varepsilon)}(z)}{\varepsilon^d}dz\\
			=&\left(|g(s)|^{s_2} * J_\varepsilon\right)(x),
		\end{aligned}
	\end{equation}
	where $J_\varepsilon(z)=\frac{{\bf 1}_{B(0,\varepsilon)}(z)}{\varepsilon^d} \geq 0$ and $\|J_\varepsilon\|_{L^1{(\mathbb{ R}^{d})}}=$measure of $B(0,1)$.
	Then substituting \eqref{l4} and \eqref{l5} into \eqref{l3}, we infer
	\begin{equation}\label{l6}
		I_1\leq C\int_{t-\varepsilon}^{t+\varepsilon}\frac{1}{\varepsilon}\left(\int_0^1 |\partial_t f(t+\tau(s-t))|^{s_1}d\tau* J_\varepsilon(x)\right)^{\frac{1}{s_1}}\Big(|g(s)|^{s_2} * J_\varepsilon(x)\Big)^{\frac{1}{s_2}}ds.
	\end{equation}
	Then using the Minkowski inequality and H\"older's inequality, one can derive that
		\begin{align}
			&\|I_1\|_{L^q(\Omega)}\\
			\leq& C\B\|  \int_{t-\varepsilon}^{t+\varepsilon}\frac{1}{\varepsilon}\left(\int_0^1 |\partial_t f(t+\tau(s-t))|^{s_1}d\tau* J_\varepsilon(x)\right)^{\frac{1}{s_1}}\Big(|g(s)|^{s_2} * J_\varepsilon(x)\Big)^{\frac{1}{s_2}}ds\B\|_{L^q(\Omega)}\nonumber \\
			\leq& C	 \int_{t-\varepsilon}^{t+\varepsilon}\frac{1}{\varepsilon}\B\|\left(\int_0^1 |\partial_t f(t+\tau(s-t))|^{s_1}d\tau* J_\varepsilon(x)\right)^{\frac{1}{s_1}}\Big(|g(s)|^{s_2} * J_\varepsilon(x)\Big)^{\frac{1}{s_2}}\B\|_{L^q(\Omega)}ds\nonumber\\
			\leq& C \int_{t-\varepsilon}^{t+\varepsilon}\frac{1}{\varepsilon}\B\|\left(\int_0^1 |\partial_t f(t+\tau(s-t))|^{s_1}d\tau* J_\varepsilon(x)\right)^{\frac{1}{s_1}}\B\|_{L^{q_1}(\Omega)}\B\|\Big(|g(s)|^{s_2} * J_\varepsilon(x)\Big)^{\frac{1}{s_2}}\B\|_{L^{q_2}(\Omega)}ds\nonumber\\
			\leq& C \int_{t-\varepsilon}^{t+\varepsilon}\frac{1}{\varepsilon}\left(\int_0^1\|\partial_t f(t+\tau (s-t))\B\|_{L^{q_1}(\Omega)}^{s_1}d \tau \right)^{\f{1}{s_1}}\|g(s)\|_{L^{q_2}(\Omega)}ds\nonumber\\
			\leq& C \left(\int_{t-\varepsilon}^{t+\varepsilon}\frac{1}{\varepsilon}\int_0^1\|\partial_t f(t+\tau (s-t))\|_{L^{q_1}(\Omega)}^{s_1}d \tau ds\right)^{\frac{1}{s_1}}\left(\int_{t-\varepsilon}^{t+\varepsilon}\frac{1}{\varepsilon}\|g(s)\|_{L^{q_2}(\Omega)}^{{s_2}}ds\right)^{\frac{1}{s_2}}\nonumber\\
			\leq & C\left(\int_0^1\frac{1}{\tau\varepsilon}\int_{-\tau\varepsilon}^{\tau\varepsilon}\|\partial_t f(t-\varsigma)\|_{L^{q_1}(\Omega)}^{s_1} d\varsigma d \tau  \right)^{\frac{1}{s_1}}\left(\int_{-\varepsilon}^{\varepsilon}\frac{1}{\varepsilon}\|g(t-\varsigma)\|_{L^{q_2}(\Omega)}^{s_2}d\varsigma\right)^{\frac{1}{s_2}}\nonumber\\
			\leq& C\Big(\|\partial_t f(t-\varsigma)\|_{L^{q_1}(\Omega)}^{s_1}* J_{1\varepsilon}(\varsigma)\Big)^{\frac{1}{s_1}}\Big(\|g(t-\varsigma)\|_{L^{q_2}(\Omega)}^{s_2}*J_{2\varepsilon}(\varsigma)\Big)^{\frac{1}{s_2}}\label{17},
		\end{align}	
	where $J_{1\varepsilon}=\int_0^1 \frac{{\bf 1}_{(-\tau\varepsilon, \tau\varepsilon)}}{\tau\varepsilon}(\varsigma)d\tau \geq 0,\ J_{2\varepsilon}=\frac{{\bf 1}_{(-\varepsilon,\varepsilon)}}{\varepsilon}(\varsigma)\geq 0$ and $\|J_{1\varepsilon}\|_{L^1{(\mathbb{R})}}=\|J_{2\varepsilon}\|_{L^1{(\mathbb{R})}}=2$.
	
	Using the H\"older's inequality again, $1\leq s_1\leq p_1$, and\ $1\leq s_2\leq p_2$, this immediately gives
	\begin{equation}\label{l8}
		\begin{aligned}
			&\|I_1\|_{L^p(L^q)}\\\leq &C \left\|\Big(\|\partial_t f(t-\varsigma)\|_{L^{q_1}(\mathbb{T}^d)}^{s_1}* J_{1\varepsilon}(\varsigma)\Big)^{\frac{1}{s_1}}\right\|_{L^{p_1}(0,T)}\left\|\Big(\|g(t-\varsigma)\|_{L^{q_2}(\mathbb{T}^d)}^{s_2}*J_{2\varepsilon}(\varsigma)\Big)^{\frac{1}{s_2}}\right\|_{L^{p_2}(0,T)}\\
\leq& C\|\partial_t f\|_{L^{p_1}(L^{q_1})}\|g\|_{L^{p_2}(L^{q_2})}.
		\end{aligned}
	\end{equation}
	As the same manner of the derivation above, we see that
	\begin{equation}\label{l9}
		\begin{aligned}
			I_2&\leq C\int_{t-\varepsilon}^{t+\varepsilon}\frac{1}{\varepsilon}\left(\int_{B(x,\varepsilon)}\frac{1}{\varepsilon^d}\int_0^1|\nabla f(t, x+\tau(y-x))|^{s_1} d\tau dy\right)^{\frac{1}{s_1}}\left(\int_{B(x,\varepsilon)}\frac{1}{\varepsilon^d}|g(s,y)|^{s_2}dy \right)^{\frac{1}{s_2}}ds\\
			&\leq C\int_{t-\varepsilon}^{t+\varepsilon}\frac{1}{\varepsilon}\left(|\nabla f(t,x-z)|^{s_1}* \overline{J}_{\varepsilon}(z)\right)^{\frac{1}{s_1}}\left(|g(s)|^{s_2}*J_\varepsilon(z)\right)^{\frac{1}{s_2}}ds
		\end{aligned}
	\end{equation}
	where $\overline{J}_{\varepsilon}(z)=\frac{{\bf 1}_{B(0,\tau\varepsilon)}}{(\tau\varepsilon)^d}\geq 0$ and $\|\overline{J}_\varepsilon(z)\|_{L^1(\mathbb{R}^{d})}=$ measure of ball $B(0,1)$.
	
	Hence, we further get
	\begin{align}
		\|I_2\|_{L^q(\Omega)}\leq& C\B\| \int_{t-\varepsilon}^{t+\varepsilon}\frac{1}{\varepsilon}\left(|\nabla f(t,x-z)|^{s_1}* \overline{J}_{\varepsilon}(z)\right)^{\frac{1}{s_1}}\left(|g(s,x-z)|^{s_2}*J_\varepsilon(z)\right)^{\frac{1}{s_2}}ds\B\|_{L^q(\Omega)}\nonumber\\
		\leq & C\int_{t-\varepsilon}^{t+\varepsilon}\frac{1}{\varepsilon}\|\left(|\nabla f(t,x-z)|^{s_1}* \overline{J}_{\varepsilon}(z)\right)^{\frac{1}{s_1}}\left(|g(s,x-z)|^{s_2}*J_\varepsilon(z)\right)^{\frac{1}{s_2}}\|_{L^q(\Omega)} ds\nonumber\\
		\leq & C\int_{t-\varepsilon}^{t+\varepsilon}\frac{1}{\varepsilon}\|\nabla f(t)\|_{L^{q_1}(\Omega)}\|g(s)\|_{L^{q_2}(\Omega)} ds\nonumber\\
		\leq &C\|\nabla f(t)\|_{L^{q_1}(\Omega)} \int_{t-\varepsilon}^{t+\varepsilon}\frac{1}{\varepsilon}\|g(s)\|_{L^{q_2}(\Omega)} ds\nonumber\\
		\leq	& C\|\nabla f(t)\|_{L^{q_1}(\Omega)} \left(\|g(t-\varsigma)\|_{L^{q_2}(\Omega)} *J_{2\varepsilon}(\varsigma)\right),\label{l10}
	\end{align}	
	which implies that
	\begin{equation}\label{l11}
		\begin{aligned}
			\|I_2\|_{L^p(L^q)}&\leq C\B\|\|\nabla f(t)\|_{L^{q_1}(\Omega)}\left(\|g(t-\varsigma)\|_{L^{q_2}(\Omega)} *J_{2\varepsilon}(\varsigma)\right)\B\|_{L^p(0,T)}\\
			&\leq C\|\nabla f\|_{L^{p_1}(L^{q_1})}\left\|\left(\|g(t-\varsigma)\|_{L^{q_2}(\Omega)} *J_{2\varepsilon}(\varsigma)\right)\right\|_{L^{p_2}}\\
			&\leq C\|\nabla f\|_{L^{p_1}(L^{q_1})}\|g\|_{L^{p_2}(L^{q_2})}.
		\end{aligned}
	\end{equation}
	Then together with \eqref{l1}, \eqref{l8} and \eqref{l11}, we  obtain that
	\begin{equation}
		\begin{aligned}
			\|\partial(fg)^\varepsilon-\partial(f g^\varepsilon)\|_{L^p(L^q)}&\leq C\left(\|G^\varepsilon\|_{L^p(L^q)}+\|\partial f g^\varepsilon\|_{L^p(L^q)}\right)\\
			&\leq C\left(\|\partial_t f\|_{L^{p_1}(L^{q_1})}+\|\nabla f\|_{L^{p_1}(L^{q_1})}\right)\|g\|_{L^{p_2}(L^{q_2})}.
		\end{aligned}
	\end{equation}
	Furthermore, if $1\leq p_2,q_2<\infty$, let $\{g_n\}\in C_{0}^\infty (\Omega)$ with $ g_n\rightarrow g$ strongly in $L^{p_2}(L^{q_2})$. Thus, by the density arguments and properties of the standard mollification, we
	arrive at
	\begin{equation}
		\begin{aligned}
			&\|\partial(fg)^\varepsilon-\partial(f g^\varepsilon)\|_{L^p(L^q)}\\\leq& C \|\partial \left(f\left(g-g_n\right)\right)^\varepsilon-\partial \left(f\left(g-g_n\right)^\varepsilon\right)+\left(\partial \left(f g_n\right)^\varepsilon-\partial\left(f g_n^\varepsilon\right)\right)\|_{L^p(L^q)}\\
			\leq& C\left(\|\partial_t f\|_{L^{p_1}(L^{q_1})}+\|\nabla f\|_{L^{p_1}(L^{q_1})}\right)\|g-g_n\|_{L^{p_2}(L^{q_2})}+C\|\partial (f g_n)^\varepsilon-\partial (f g_n)\|_{L^{p}(L^{q})}\\
			&\ \ \ \ + C\|\partial (f (g_n-g_n ^\varepsilon))\|_{L^{p}(L^{q})}\rightarrow 0,\ \ as \ \varepsilon\rightarrow 0.
		\end{aligned}
	\end{equation}
 	This concludes the proof.
 \end{proof}
 \begin{lemma}\label{lem2.2}
	Let $ p,q,p_1,q_1,p_2,q_2\in[1,+\infty)$ with $\frac{1}{p}=\frac{1}{p_1}+\frac{1}{p_2},\frac{1}{q}=\frac{1}{q_1}+\frac{1}{q_2} $. Assume $f\in L^{p_1}(0,T;L^{q_1}(\Omega)) $ and $g\in L^{p_2}(0,T;L^{q_2}(\Omega))$. Then for any $\varepsilon>0$, there holds
	\begin{equation}\label{a4}
	\|(fg)^\varepsilon-f^\varepsilon g^\varepsilon\|_{L^p(0,T;L^q(\Omega))}\rightarrow 0,\ \ \ as\ \varepsilon\rightarrow 0.
	\end{equation}
\end{lemma}
\begin{proof}
	By the triangle inequality, one have
	\begin{equation*}
	\begin{aligned}
	&\|(fg)^\varepsilon-f^\varepsilon g^\varepsilon\|_{L^p(0,T;L^q(\Omega))}\\
	\leq & C\left(\|(fg)^\varepsilon- (fg)\|_{L^p(0,T;L^q(\Omega))}+\|fg-f^\varepsilon g\|_{L^p(0,T;L^q(\Omega))}+\|f^\varepsilon g-f^\varepsilon g^\varepsilon\|_{L^p(0,T;L^q(\Omega))}\right)\\
	\leq &C\Big(\|(fg)^\varepsilon- fg\|_{L^p(0,T;L^q(\Omega))}+\|f-f^\varepsilon\|_{L^{p_1}(0,T;L^{q_1}(\Omega))}\|g\|_{L^{p_2}(0,T;L^{q_2}(\Omega))}\\
	&\ \ \ \ \ +\|f^\varepsilon\|_{L^{p_1}(0,T;L^{q_1}(\Omega))}\|g-g^\varepsilon\|_{L^{p_2}(0,T;L^{q_2}(\Omega))}\Big),
	\end{aligned}
	\end{equation*}
	then, together with the properties of the standard mollification, we can obtain \eqref{a4}.
\end{proof}
Finally, we recall the generalized Aubin-Lions Lemma to extend the energy equality up to the initial time.
\begin{lemma}[\cite{[Simon]}]\label{AL}
	Let $X\hookrightarrow B\hookrightarrow Y$ be three Banach spaces with compact imbedding $X \hookrightarrow\hookrightarrow Y$. Further, let there exist $0<\theta <1$ and $M>0$ such that
	\begin{equation}\label{le1}
	\|v\|_{B}\leq M\|v\|_{X}^{1-\theta}\|v\|_{Y}^\theta\ \ for\ all\ v\in X\cap Y.\end{equation}
Denote for $T>0$,
\begin{equation}\label{le2}
	W(0,T):=W^{s_0,r_0}((0,T), X)\cap W^{s_1,r_1}((0,T),Y)
\end{equation}
with
\begin{equation}\label{le3}
	\begin{aligned}
		&s_0,s_1 \in \mathbb{R}; \ r_0, r_1\in [1,\infty],\\
		s_\theta :=(1-\theta)s_0&+\theta s_1,\ \f{1}{r_\theta}:=\f{1-\theta}{r_0}+\f{\theta}{r_1},\ s^{*}:=s_\theta -\f{1}{r_\theta}.
	\end{aligned}
\end{equation}
Assume that $s_\theta>0$ and $F$ is a bounded set in $W(0,T)$. Then, we have

If $s_{*}\leq 0$, then $F$ is relatively compact in $L^p((0,T),B)$ for all $1\leq p< p^{*}:=-\f{1}{s^{*}}$.

If $s_{*}> 0$, then $F$ is relatively compact in $C((0,T),B)$.

\end{lemma}

\section{Proof of Theorem $\ref{the1.1}$, Theorem \ref{the1.2}   and  Corollary \ref{coro1.2}}
In this section, we first present the proof of Theorem \ref{the1.2}. Then,
making  the use of interpolation inequality and the Poincar\'e inequality, we prove  the Theorem \ref{the1.1} and  Corollary \ref{coro1.2}  by the results of Theorem \ref{the1.2}.
  \begin{proof}[Proof of Theorem \ref{the1.2}]
Let $\phi(t)$ be a smooth function compactly supported in $(0,+\infty)$.
Multiplying $\eqref{INS}_2$ by $(\phi v^{\varepsilon})^\varepsilon$, then integrating over $(0,T)\times \Omega$ , we infer that
\begin{equation}\label{c1} \begin{aligned}
		\int_0^T\int \phi(t)v^{\varepsilon}\B[\partial_{t}(\rho v)^{\varepsilon}+ \Div(\rho v\otimes v)^{\varepsilon}+\nabla P(\rho)^\varepsilon- \mu \Delta v^\varepsilon-(\mu+\lambda)\nabla( \Div v)^\varepsilon\B]=0.
\end{aligned}\end{equation}
To pass the limit of $\varepsilon$, we
reformulate  every term of the last equation.
A  straightforward  computation leads to
\begin{equation}\label{c2}
	\begin{aligned}
		\int_0^T\int \phi(t)v^{\varepsilon} \partial_{t} (\rho v )^{\varepsilon}=&\int_0^T\int \phi(t)v^{\varepsilon}\B[ \partial_{t} (\rho v )^{\varepsilon}-\partial_{t}(\rho v^{\varepsilon})\B]+ \int_0^T\int \phi(t)v^{\varepsilon} \partial_{t}(\rho v^{\varepsilon}) \\
		=& \int_0^T\int \phi(t)v^{\varepsilon} \B[\partial_{t} (\rho v )^{\varepsilon}-\partial_{t}(\rho v^{\varepsilon})\B]+\int_0^T\int \phi(t)\rho\partial_t{\frac{|v^{\varepsilon}|^2}{2}} \\
		&+\int_0^T\int \phi(t)\rho_t|v^{\varepsilon}|^2.
\end{aligned}\end{equation}
It follows from integration by parts and the mass equation $\eqref{INS}_1$ that
\begin{align}
	&\int_0^T\int\phi(t)v^{\varepsilon} \Div(\rho v\otimes v)^{\varepsilon}\nonumber\\
	=& \int_0^T\int\phi(t)  v^{\varepsilon}  \Div[(\rho v\otimes v)^{\varepsilon}-(\rho  v)\otimes v^{\varepsilon}]+\int_0^T\int\phi(t)v^{\varepsilon}\Div(\rho  v\otimes v^{\varepsilon})\nonumber\\
	=& -\int_0^T\int\phi(t)  \nabla v^{\varepsilon}  [(\rho v\otimes v)^{\varepsilon}-(\rho  v)\otimes v^{\varepsilon}]+  \int_0^T\int \phi(t)\left(\Div (\rho v ) |v^{\varepsilon}|^{2}+\f12 \rho v \nabla|v^{\varepsilon}  |^{2}
	\right)\nonumber\\
	=& -\int_0^T\int\phi(t) \nabla  v^{\varepsilon}  [(\rho v\otimes v)^{\varepsilon}-(\rho  v)\otimes v^{\varepsilon}]+\f{1}{2}\int_0^T\int\phi(t) \Div (\rho v ) |v^{\varepsilon}|^{2}\nonumber\\
	=& -\int_0^T\int\phi(t)  \nabla v^{\varepsilon}  [(\rho v\otimes v)^{\varepsilon}-(\rho  v)\otimes v^{\varepsilon}]-\frac{1}{2}\int_0^T\int \phi(t) \partial_t \rho |v^{\varepsilon}|^{2}. \label{c3}
\end{align}
We rewrite the pressure term as
\begin{equation}\label{c6}
	\begin{aligned}
		&\int_0^T\int\phi(t)v^{\varepsilon}\nabla (\rho^\gamma)^{\varepsilon}= \int_0^T\int\phi(t) [v^{\varepsilon}\nabla(\rho^\gamma)^{\varepsilon}-v\nabla (\rho^\gamma)]+\int_0^T\int\phi(t)v \nabla (\rho^\gamma).
\end{aligned} \end{equation}
Then using the integration by parts and mass equation $\eqref{INS}_1$ again, we find
$$\ba
\int_0^T\int\phi(t)v \cdot\nabla (\rho^\gamma)
=&-\int_0^T\int\phi(t) \rho^{\gamma-1} \rho\text{div\,}v\\
=&\int_0^T\int\phi(t) \rho^{\gamma-1}(\partial_{t}\rho+v\cdot\nabla\rho)
\\
=&\f{1}{\gamma}\int_0^T\int\phi(t)  \partial_{t}\rho^{\gamma } +\f{1}{\gamma}\int_0^T\int\phi(t)v\cdot\nabla\rho^{\gamma },
\ea $$
which in turn means that
\begin{equation}\label{c62}\begin{aligned}
		\int_0^T\int\phi(t)v \cdot\nabla (\rho^\gamma)
		= \f{1}{\gamma-1}\int_0^T\int\phi(t)  \partial_{t}\rho^{\gamma }.
\end{aligned}\end{equation}
Thanks to integration by parts, we arrive at
\begin{equation}\label{c8}\begin{aligned} &-\mu\int_0^T\int \phi(t)v^{\varepsilon} \Delta v^{\varepsilon}=\mu \int_0^T\int \phi(t)|\nabla v^{\varepsilon}|^{2},\\
		&- (\mu+\lambda)\int_0^T\int \phi(t)v^{\varepsilon} \nabla\text{div\,}v^\varepsilon=(\mu+\lambda)\int_0^T\int \phi(t)|\Div{v^{\varepsilon}}|^{2}.
\end{aligned}  \end{equation}
Plugging  \eqref{c2}-\eqref{c8} into \eqref{c1} and using the integration by parts, we conclude  that
\begin{equation}\label{c9}
	\begin{aligned}
		&-\int_0^T\int \phi(t)_t\left(\rho{\frac{|v^{\varepsilon}|^2}{2}}+\f{1}{\gamma-1}\rho^{\gamma }\right)+\int_0^T\int \left(\mu |\nabla v^\varepsilon|^2+(\mu+\lambda)|\Div v^\varepsilon|^2\right)\\
		=&-\int_0^T\int \phi(t)v^{\varepsilon} \B[\partial_{t} (\rho v )^{\varepsilon}-\partial_{t}(\rho v^{\varepsilon})\B]+\int_0^T\int\phi(t) \nabla v^{\varepsilon}  [(\rho v\otimes v)^{\varepsilon}-(\rho  v)\otimes v^{\varepsilon}]\\ &-\int_0^T\int\phi(t) [v^{\varepsilon}\nabla(\rho^\gamma)^{\varepsilon}-v\nabla (\rho^\gamma)].
\end{aligned}\end{equation}
It is enough  to prove that the terms on the right hand-side of \eqref{c9} tend to zero as $\varepsilon\rightarrow 0$.
In view of H\"older's inequality and Lemma \eqref{pLions}, we know that
\begin{equation}\label{3.81}\begin{aligned}
\int_s^t\int \phi(t)v^{\varepsilon} \B[\partial_{t} (\rho v )^{\varepsilon}-\partial_{t}(\rho v^{\varepsilon})\B]&\leq C\|v^{\varepsilon}\|_{L^{p}(L^{q})}
\|\partial_{t} (\rho v )^{\varepsilon}-\partial_{t}(\rho v^{\varepsilon})\|_{L^{\f{p}{p-1}}(L^{\f{q}{q-1}})}\\
&\leq C\|v\|^{2}_{L^{p}(L^{q})}
\left(\|\rho_{t}\|_{L^{\f{p}{p-2}}(L^{\f{q}{q-2}})}+\|\nabla \rho\|_{L^{\f{p}{p-2}}(L^{\f{q}{q-2}})}\right).
\end{aligned}\end{equation}
To bound $\rho_{t}$ and $\nabla \rho$, we
employ  mass equation to obtain
$$
\rho_t=-2\sqrt{\rho}v\cdot\nabla\sqrt{\rho}-{\rho}\text{div}v, \ and\ \nabla \rho=2\sqrt{\rho}\nabla \sqrt{\rho}. $$
As a consequent, the triangle inequality and H\"older's inequality  guarantee that
\begin{equation}\label{3.91}\begin{aligned}
\|\rho_{t}\|_{L^{\f{p}{p-2}}(L^{\f{q}{q-2}})}&\leq
 C\left(\|-2\sqrt{\rho}v\cdot\nabla\sqrt{\rho} \|_{L^{\f{p}{p-2}}(L^{\f{q}{q-2}})}+
\|{\rho}\text{div}v\|_{L^{\f{p}{p-2}}(L^{\f{q}{q-2}})}\right)\\
&\leq C\left(
\|v \|_{L^{p}(L^{q})}\|\nabla\sqrt{\rho} \|_{L^{\f{p}{p-3}}(L^{\f{q}{q-3}})}+\| \nabla v\|_{L^{\f{p}{p-2}}(L^{\f{q}{q-2}})}\right),
 \end{aligned}\end{equation}
and
\begin{equation}\label{3.92}
	\begin{aligned}
	\|\nabla \rho\|_{L^{\f{p}{p-2}}(L^{\f{q}{q-2}})}\leq & C\|\sqrt{\rho}\nabla \sqrt{\rho}\|_{L^{\f{p}{p-2}}(L^{\f{q}{q-2}})}\leq C\|\nabla \sqrt{\rho}\|_{L^{\f{p}{p-3}}(L^{\f{q}{q-3}})}.
	\end{aligned}
\end{equation}
Plugging \eqref{3.91} and \eqref{3.92} into \eqref{3.81}, we get
 \begin{equation}\label{c3.11}\begin{aligned}
&\int_0^T\int \phi(t)v^{\varepsilon} \B[\partial_{t} (\rho v )^{\varepsilon}-\partial_{t}(\rho v^{\varepsilon})\B]\\\leq&
C\|v\|^{2}_{L^{p}(L^{q})}
\left(\left(\|v \|_{L^{p}(L^{q})}+1\right)\|\nabla\sqrt{\rho} \|_{L^{\f{p}{p-3}}(L^{\f{q}{q-3}})}+\| \nabla v\|_{L^{\f{p}{p-2}}(L^{\f{q}{q-2}})}\right)\leq C.
 \end{aligned}\end{equation}
 From Lemma \eqref{pLions}, we end up with, as $\varepsilon\rightarrow0$,
 $$
 \int_0^T\int \phi(t)v^{\varepsilon} \B[\partial_{t} (\rho v )^{\varepsilon}-\partial_{t}(\rho v^{\varepsilon})\B]\rightarrow0.$$
  Using the integration by parts, we observe that
 \begin{equation}\label{3.11}\begin{aligned}
 		&\left|\int_0^T\int\phi(t)  \nabla v^{\varepsilon}  [(\rho v\otimes v)^{\varepsilon}-(\rho  v)\otimes v^{\varepsilon}]\right| \\
 		\leq &C\|\nabla v^\varepsilon\|_{L^{\f{p}{p-2}}(L^{\f{q}{q-2}})}\|(\rho v\otimes v)^{\varepsilon}-(\rho  v)\otimes v^{\varepsilon}\|_{L^{\f{p}{2}}(L^{\f{q}{2}})}\\
 		\leq & C\|\nabla v\|_{L^{\f{p}{p-2}}(L^{\f{q}{q-2}})}\left(\|(\rho v \otimes v)^\varepsilon- \rho v \otimes v\|_{L^{\f{p}{2}}(L^{\f{q}{2}})}+\|\rho v\otimes v- \rho v\otimes v^\varepsilon\|_{L^{\f{p}{2}}(L^{\f{q}{2}})}\right)\\
 		\leq & C\|\nabla v\|_{L^{\f{p}{p-2}}(L^{\f{q}{q-2}})} \left(\|(\rho v \otimes v)^\varepsilon- \rho v \otimes v\|_{L^{\f{p}{2}}(L^{\f{q}{2}})}+\|\rho v\|_{L^p(L^q)}\|v-v^\varepsilon\|_{L^p(L^q)}\right).
 \end{aligned}\end{equation}
 Hence, by the standard properties of the mollification, we have
 $$\int_0^T\int\phi(t)  \nabla v^{\varepsilon}  [(\rho v\otimes v)^{\varepsilon}-(\rho  v)\otimes v^{\varepsilon}] \rightarrow0 \text{ as }\varepsilon\rightarrow0.$$
 According to  the upper bound of the density, H\"older's inequality on bound domain, we observe that
 $$\|\nabla(\rho^\gamma)\|_{L^{\f{p}{p-1}}(L^{\f{q}{q-1}})}\leq
 C \|\nabla \sqrt{\rho} \|_{L^{\f{p}{p-1}}(L^{\f{q}{q-1}})}\leq C\|\nabla\sqrt{\rho} \|_{L^{\f{p}{p-3}}(L^{\f{q}{q-3}})},$$
 which in turn implies that
 \begin{equation}\label{c10}\begin{aligned}
 \int_0^T\int\phi(t) [v^{\varepsilon}\nabla(\rho^\gamma)^{\varepsilon}-v\nabla (\rho^\gamma)]\rightarrow0,
 \end{aligned}\end{equation}
 where we have used lemma \ref{lem2.2}.

Then together with \eqref{c3.11}-\eqref{c10}, passing to the limits as $\varepsilon\rightarrow 0$, we know that
\begin{equation}\label{c11}
	\begin{aligned}
		-\int_0^T\int \phi_t \left(\frac{1}{2}\rho |v|^2+\frac{\rho^\gamma}{\gamma -1}\right)+\int_0^T\int\phi(t) \left(\mu |\nabla v|^2+(\mu+\lambda)|\Div v|^2\right)=0.	
	\end{aligned}
\end{equation}
The next objective is to get the energy equality up to the initial time $t=0$ by the similar method in \cite{[CLWX]} and \cite{[Yu2]}, for the convenience of the reader and the integrity of the paper, we give the details. First, we derive  from the mass equation $\eqref{INS}_1$ that
\begin{equation}\label{c12}
	\begin{aligned}
		\partial_t(\rho^\gamma )=-\gamma \rho ^\gamma \Div v-2\gamma \rho^{\gamma -\frac{1}{2}}v\cdot \nabla \sqrt{\rho},	
	\end{aligned}
\end{equation}
and
\begin{equation}\label{c13}
	\begin{aligned}
		\partial_t(\sqrt{\rho})=\frac{\sqrt{\rho }}{2}\Div v+v\cdot \nabla \sqrt{\rho},
	\end{aligned}
\end{equation}
which together with \eqref{key0} gives
$$\partial_t(\rho^\gamma, \sqrt{\rho})\in L^{\f{p}{p-2}}(0,T;L^{\f{q}{q-2}}(\Omega)),\ \ \ \nabla (\rho^\gamma, \sqrt{\rho })\in L^{\f{p}{p-3}}(0,T; L^{\f{q}{q-3}}(\Omega)).$$
Hence, using the Aubin-Lions Lemma \ref{AL}, we can obtain
\begin{equation}\label{c14}
	(\rho^\gamma ,\ \sqrt{\rho })\in C([0,T];L^{\f{q}{q-3}}(\Omega)),\ for \ dp< 2q+3d,\ p\geq 4\ and\ q\geq 4.
\end{equation}
Furthermore, by the moumentum equality $\eqref{INS}_2$, we know that
$$\rho v\in L^\infty(0,T;L^2(\Omega)) \cap H^1(0,T; H^{-1}(\Omega)),$$
then, in light of  the Aubin-Lions Lemma \ref{AL}, we conclude that
\begin{equation}\label{c15}
	\rho v\in C([0,T]; L^2_{weak}(\Omega)).
\end{equation}
Meanwhile, using the natural energy \eqref{energyineq}, \eqref{c14} and \eqref{c15}, we have
\begin{equation}\label{c16}
	\begin{aligned}
		0&\leq \overline{\lim_{t\rightarrow 0}}\int |\sqrt{\rho} v-\sqrt{\rho_0}v_0|^2 dx\\
		&=2\overline{\lim_{t\rightarrow 0}}\left(\int \left(\f{1}{2}\rho |v|^2 +\f{1}{\gamma -1}\rho ^\gamma \right)dx-\int\left(\f{1}{2}\rho_0 |v_0|^2+\f{1}{\gamma -1}\rho_0 ^\gamma \right)dx\right)\\
		&\ \ \ +2\overline{\lim_{t\rightarrow 0}}\left(\int\sqrt{\rho_0}v_0\left(\sqrt{\rho_0}v_0-\sqrt{\rho} v\right)dx+\f{1}{\gamma -1}\int \left(\rho_0^\gamma -\rho^\gamma\right)dx\right)\\
		&\leq 2\overline{\lim_{t\rightarrow 0}}\int \sqrt{\rho_0}v_0\left(\sqrt{\rho_0}v_0-\sqrt{\rho}v\right)dx\\
		&=2\overline{\lim_{t\rightarrow 0}}\int v_0 \left(\rho_0 v_0 -\rho v\right)dx+2\overline{\lim_{t\rightarrow 0}}\int v_0 \sqrt{\rho }v\left(\sqrt{\rho }-\sqrt{\rho_0}\right)dx=0,
	\end{aligned}
\end{equation}
from
which it follows
\begin{equation}\label{c17}
	\sqrt{\rho} v(t)\rightarrow \sqrt{\rho }v(0)\ \ strongly\ in\ L^2(\Omega)\ as\ t\rightarrow 0^+.
\end{equation}
Similarly, one has the right temporal continuity  of $\sqrt{\rho}v$ in $L^2(\Omega)$, hence, for any $t_0\geq 0$, we infer that
\begin{equation}\label{c18}
	\sqrt{\rho} v(t)\rightarrow \sqrt{\rho }v(t_0)\ \ strongly\ in\ L^2(\Omega)\ as\ t\rightarrow t_0^+.
\end{equation}
Before we go any further, it should be noted that \eqref{c11} remains valid for function $\phi$ belonging to $W^{1,\infty}$ rather than $C^1$, then for any $t_0>0$, we redefine the test function $\phi$ as $\phi_\tau$ for some positive $\tau$ and $\alpha $ such that $\tau +\alpha <t_0$, that is
\begin{equation}
	\phi_\tau(t)=\left\{\begin{array}{lll}
		0, & 0\leq t\leq \tau,\\
		\f{t-\tau}{\alpha}, & \tau\leq t\leq \tau+\alpha,\\
		1, &\tau+\alpha \leq t\leq t_0,\\
		\f{t_0-t}{\alpha }, & t_0\leq t\leq t_0 +\alpha ,\\
		0, & t_0+\alpha \leq t.
	\end{array}\right.
\end{equation}
Then substituting this test function into \eqref{c11}, we arrive at
\begin{equation}
	\begin{aligned}
		-\int_\tau^{\tau+\alpha}\int& \f{1}{\alpha}\left(\f{1}{2}\rho v^2+\f{1}{\gamma-1}\rho^\gamma \right)+\f{1}{\alpha}\int_{t_0}^{t_0+\alpha}\int 	\left(\f{1}{2}\rho v^2+\f{1}{\gamma-1}\rho^\gamma \right)\\
		&+\int_{\tau}^{t_0+\alpha}\int \phi_\tau \left(\mu |\nabla v|^2+\left(\mu+\lambda\right)|\Div v|^2\right)=0.
	\end{aligned}	
\end{equation}
Taking $\alpha\rightarrow 0$ and using  the fact that $\int_0^t\int\left(\mu |\nabla v|^2+\left(\mu+\lambda\right)|\Div v|^2\right)$ is continuous with respect to $t$ and the  Lebesgue point Theorem, we deduce that
\begin{equation}
	\begin{aligned}
		-\int&\left(\f{1}{2}\rho v^2+\f{1}{\gamma-1}\rho^\gamma \right)(\tau)dx+\int\left(\f{1}{2}\rho v^2+\f{1}{\gamma-1}\rho^\gamma \right)(t_0)dx\\
		&+\int_\tau^{t_0}\int\left(\mu |\nabla v|^2+\left(\mu+\lambda\right)|\Div v|^2\right)dxds=0.
	\end{aligned}
\end{equation}
Finally, letting $\tau\rightarrow 0$, using the continuity of $\int_0^t\int\left(\mu |\nabla v|^2+\left(\mu+\lambda\right)|\Div v|^2\right)$, \eqref{c14} and \eqref{c17}, we can obtain
\begin{equation}\ba
	\int\left(\f{1}{2}\rho v^2+\f{1}{\gamma-1}\rho^\gamma \right)(t_0)dx+\int_0^{t_0}\int&\left(\mu |\nabla v|^2+\left(\mu+\lambda\right)|\Div v|^2\right)dxds\\=&\int\left(\f{1}{2}\rho_0 v_0^2+\f{1}{\gamma-1}\rho_0^\gamma \right)dx.
	\ea\end{equation}
Then we complete the proof of Theorem \ref{the1.2}.
\end{proof}
Next, with the help of  Theorem \ref{the1.2},
we are in a position to prove Theorem \ref{the1.1}.
 \begin{proof}[Proof of Theorem \ref{the1.1}]
(1)
Choose $p=q=4$ in \eqref{key0}, we obtain Lions  type  energy
  conservation condition.
\be\ba\label{spe1}
&0\leq \rho<c<\infty, \nabla\sqrt{\rho}\in L^{4}(L^{4}),\\& v\in L^{4}(L^{4}), \nabla v\in L^{2}(L^{2})\ and \ v_0\in L^2(\mathbb{T}^d).\ea\ee
(2)
With the help of interpolation, we show \eqref{vei819}  can be reduced to \eqref{vei}. Indeed, taking advantage of  the Gagliardo-Nirenberg   inequality on bounded domains, H\"older's inequality and Young's inequality, we know that
\begin{equation}
	\begin{aligned}
		\|v\|_{L^{4}(0,T;L^{4}(\mathbb{T}^d))}
		\leq & C\|v\|_{L^2(0,T;L^6(\mathbb{T}^d))}^{\frac{3(4-q)}{2(6-q)}}\|v\|_{L^p(0,T;L^q(\mathbb{T}^d))}^{\frac{q}{2(6-q)}}\\
		\leq& C\left(\|\nabla v\|_{L^2(0,T;L^2(\mathbb{T}^d))}+\|v\|_{L^p(0,T;L^q(\mathbb{T}^d))}\right)^{\frac{3(4-q)}{2(6-q)}}\|v\|_{L^p(0,T;L^q(\mathbb{T}^d))}^{\frac{q}{2(6-q)}}\\
		\leq &C\left(\|\nabla v\|_{L^2(0,T;L^2(\mathbb{T}^d))}^{\frac{3(4-q)}{2(6-q)}}\|v\|_{L^p(0,T;L^q(\mathbb{T}^d))}^{\frac{q}{2(6-q)}}
		+\|v\|_{L^p(0,T;L^q(\mathbb{T}^d))}\right).
	\end{aligned}
\end{equation}
 From the the result just proved, we obtain  energy equality via \eqref{vei}.
 We finish the proof of this case.

(3) Taking $p=4,q=6$ in \eqref{key0}, we immediately get the desired result.

(4) We move  to the proof of the last case.
Firstly,  we temporarily assume that the following fact is valid,
\be\label{assumed}
\int_{\Omega}vdx\in L^{p}(0,T).
\ee
It follows from the triangle inequality and Poincar\'e-Sobolev inequality that, for $q<d$,
$$\ba
\|v\|_{L^{p}(L^{\f{dq}{d-q}})}
\leq&\B\|v-\f{1}{|\Omega|}\int_{\Omega}vdy
\B\|_{L^{p}(L^{\f{dq}{d-q}})}+\B\|\f{1}{|\Omega|}\int_{\Omega}vdy
\B\|_{L^{p}(L^{\f{dq}{d-q}})}\\
\leq& C\|\nabla v
\|_{L^{p}(L^{q})}+C\B\| \int_{\Omega}vdy
\B\|_{L^{p} (0,T)}.
\ea$$
From the result  case (2) in  just proved, it is enough to
prove the quality \eqref{assumed} we have assumed.
The H\"older inequality and classical Poincar\'e inequality ensure that
\be\ba\label{3.28}
\B|\int_{\Omega}\rho\left(v-\f{1}{|\Omega|}\int_{\Omega}vdy\right)  dx\B|\leq& \|\rho\|_{L^\f{q}{q-1}}
\B\|v-\f{1}{|\Omega|}\int_{\Omega}vdy\B\|_{L^{q}}\\
\leq& C\|\rho\|_{L^\f{q}{q-1}}
\|\nabla v\|_{L^{q}}.
\ea\ee
Using the H\"older's inequality once again and  the upper bound of the density, we find
\be\label{3.29}
\int \rho vdx\leq C\|\sqrt{\rho}v\|_{L^{2}}.
\ee
In view of the  triangle inequality, \eqref{3.28} and \eqref{3.29}, we infer that
$$\ba
\B|\int_{\Omega} v dy\B|=&\f{|\Omega|}{\int_{\Omega} \rho_{0}dx} \B|\int_{\Omega} \left(\f{1}{|\Omega|}\int_{\Omega}vdy\right)\rho dx\B|\\
\leq&\f{C}{\int_{\Omega} \rho_{0}dx} \B|\int_{\Omega}\B[\left( \f{1}{|\Omega|}\int_{\Omega}vdy\right)-v\B]\rho dx\B|+
\f{C}{\int_{\Omega} \rho_{0}dx} \B| \int_{\Omega} \rho vdx\B|\\
\leq& C
\|\nabla v\|_{L^{q}}+C\|\sqrt{\rho}v\|_{L^{2}}.
\ea$$
which means
$$\ba
\B\|\int_{\Omega} v dy\B\|_{L^{p}(0,T)}
\leq& C
\|\nabla v\|_{L^{p}(L^{q})}+C\|\sqrt{\rho}v\|_{L^{\infty}(L^{2})}.
\ea$$
Thus, it follows from the triangle inequality and Poincar\'e-Sobolev inequality that
$$\ba
\|v\|_{L^{p}(L^{q})}
\leq&\B\|v-\f{1}{|\Omega|}\int_{\Omega}vdy
\B\|_{L^{p}(L^{q})}+\B\|\f{1}{|\Omega|}\int_{\Omega}vdy
\B\|_{L^{p}(L^{q})}\\
\leq& C\|\nabla v
\|_{L^{p}(L^{q})}+C\B\| \int_{\Omega}vdy
\B\|_{L^{p} (0,T)}\\
\leq & C\left( \|\nabla v\|_{L^{p}(L^{q})}+\|\sqrt{\rho}v\|_{L^{\infty}(L^{2})}\right),
\ea$$
which means $v\in L^p(D^{1,q})$ implies $v\in L^p(W^{1,q})$.
which implies $$v\in L^p(0,T;W^{1,q}(\mathbb{T}^d)),\ q<d.$$
In summary,  we have shown
$$\nabla v\in L^{p}(0,T;L^{r}(\mathbb{T}^{d}))  ~\text{with}~ \f{1}{p}+
\f{3}{r}=1+\f{3}{d}, \ \f{3d}{d+3}<r\leq \f{4d}{d+4},$$
turns out
$$v\in L^{p}(0,T;L^{q}(\mathbb{T}^{d}))  ~\text{with}~ \f{1}{p}+
\f{3}{q}=1, 3<q\leq 4.$$

At this stage, we complete the proof of Theorem \ref{the1.1}.
\end{proof}
\begin{remark}
To prove \eqref{tivei}, one may directly employ the following
 Poincar\'e inequality which can be proved by slightly modifying the proof of \cite[Lemma 3.2, page 47]{[Feireisl2004]}: \\
 Let $v\in W^{1,q}(\Omega)$ with $q\geq 2$, and let $\rho $ be a non-negative function such that
	\begin{equation}\label{le4-1}
		0<M\leq \int_{\Omega}\rho dx,\ \int_{\Omega}\rho^\gamma dx\leq E_0<\infty,
	\end{equation}
where $\Omega\subset \mathbb{R}^d$ is a bounded domain and $\gamma>1$.
Then there exists a constant $c$ depending solely on $M$ and $E_0$ such that
\begin{equation}\label{le4-2}
	\|v\|_{L^q(\Omega)}^q\leq c(E_0, M)\left(\|\nabla v\|_{L^q(\Omega)}^q+\left(\int_{\Omega}\rho |v|dx\right)^q\right).
\end{equation}

\end{remark}

Next, with the help of  Theorem \ref{the1.1},  we prove Corollary \ref{coro1.2}.
\begin{proof}[Proof of  Corollary \ref{coro1.2}]
	According to Theorem \ref{the1.1}, it suffices to show that all the density hypotheses yield
	$$
	0 \leq\rho\leq c_{2}<\infty, \nabla\sqrt{\rho}\in L^{4}(0,T;L^{4}(\mathbb{T}^{3})).$$
	On the one hand, for the case (2), it is enough to show $0 \leq\rho\leq c_{2}<\infty.$
	The Sobolev embedding theorem ensures that $ \nabla\sqrt{\rho}\in L^{\infty}(0,T;L^{4}(\mathbb{T}^{3}))$ and $ {\rho}\in L^{\infty}(0,T;L^{\gamma}(\mathbb{T}^{3}))$  with $\gamma >\f{3}{2}$ (due to the theory of the  global existence of weak solutions to \eqref{INS}-\eqref{INS1} obtained by Feireisl- Novotn\'y-Petzeltov\' a in \cite{[FN]}) leads to $$0 \leq\rho\leq c_{2}<\infty.$$
	For the case (3),
	it follows from the well-know interpolation inequality and H\"older's inequality that, for $\f{1}{p_1}+\f{1}{q_1}=\f{1}{2}$, with $q_{1}\geq 4$,
	$$\ba
	\|\nabla\sqrt{\rho}\|_{L^{4}(0,T;L^{4}(\mathbb{T}^{3}) ))}  \leq& C\|\nabla\sqrt{\rho}\|_{L^{\infty}\left(0,T;L^{2}(\mathbb{T}^{3})\right)}^{\frac{(q_{1}-4)}{2(q_{1}-2)}}\|\nabla\sqrt{\rho}\|_{L^{p_1}(0,T;L^{ {q_{1} }}(\mathbb{T}^{3}))}^{\frac{ q_{1}}{2(q_{1}-2)}}<\infty.
	\ea$$
	This completes  the proof of this corollary.
\end{proof}
\section*{Acknowledgement}

Ye was partially supported by the National Natural Science Foundation of China  under grant (No.11701145) and China Postdoctoral Science Foundation (No. 2020M672196).
 Wang was partially supported by  the National Natural
 Science Foundation of China under grant (No. 11971446, No. 12071113   and  No.  11601492).
Wei was partially supported by the National Natural Science Foundation of China under grant (No. 11601423, No. 11771352, No. 11871057).

\end{document}